\documentclass{elsarticle}
\usepackage{graphicx}
\usepackage{pstricks}
\usepackage{amsmath}
\usepackage{amsfonts}
\usepackage{amssymb}
\usepackage{color}
\usepackage{psfrag}
\usepackage{subfig}
\def\R{\mathbb{R}}              

\def\L2{{\cal L}_2[0,T)}

\newcommand{\calR}{{\cal R}}  

\newcommand{\rline}{{\mathbb R}}

\newcommand{\sbmn}[1]{\left.\begin{smallmatrix} #1
   \end{smallmatrix}\right.}
\newcommand{\rfb}[1]{\mbox{\rm
   (\ref{#1})}\ifx\undefined\stillediting\else:\fbox{$#1$}\fi}
\newcommand{\dd}   {{\rm d}\hbox{\hskip 0.5pt}}

\newtheorem{theorem}{Theorem}
\newtheorem{proposition}{Proposition}
\newtheorem{lemma}{Lemma}
\newtheorem{corallary}{Corallary}
\newtheorem{definition}{Definition}
\newtheorem{remark}{Remark}
\newtheorem{example}{Example}

\newenvironment{proof}[1][Proof]{\begin{trivlist}
\item[\hskip \labelsep {\it #1.}]}{\end{trivlist}\qed}

\def\l2{{\cal L}_2[0,T)}
\def\L2e{{\cal L}_{2e}}

\def\begequarr{\begin{eqnarray*}}
\def\endequarr{\end{eqnarray*}}
\def\begarr{\begin{array}}
\def\endarr{\end{array}}
\def\begequ{\begin{equation}}
\def\endequ{\end{equation}}

\def\begdes{\begin{description}}
\def\enddes{\end{description}}
\def\begenu{\begin{enumerate}}
\def\begite{\begin{itemize}}
\def\endite{\end{itemize}}
\def\endenu{\end{enumerate}}

\def\lef[{\left[\begin{array}}
\def\rig]{\end{array}\right]}
\def\qed{\hfill$\Box \Box \Box$}
\def\begcen{\begin{center}}
\def\endcen{\end{center}}
\def\begrem{\begin{remark}\rm}
\def\endrem{\end{remark}}
\def\begrem{\begin{corallary}\rm}
\def\endrem{\end{corallary}}

\begin{document}

\begin{frontmatter}

\title{On the Characterization of the Duhem Hysteresis Operator with Clockwise Input-Output Dynamics}
\author[ruiyue]{Ruiyue Ouyang}
\ead{r.ouyang@rug.nl}
\author[vincent]{Vincent Andrieu}
\ead{vincent.andrieu@gmail.com}
\author[bayu]{Bayu Jayawardhana$^\dag$}
\ead{bayujw@ieee.org (Corresponding author)}
\address[ruiyue]{Dept. Discrete Technology and Production Automation, University of Groningen, Groningen 9747AG, The Netherlands}
\address[vincent]{Universit\'e Lyon 1, Villeurbanne; CNRS, UMR 5007, LAGEP. 43 bd du 11 novembre, 69100 Villeurbanne, France}
\address[bayu]{Dept. Discrete Technology and Production Automation, University of Groningen, Groningen 9747AG, The Netherlands}

\begin{abstract}
In this paper we investigate the dissipativity property of a certain class of Duhem hysteresis operator, which has clockwise (CW) input-output (I/O) behavior. In particular, we provide sufficient conditions on the Duhem operator such that it is CW and propose an explicit construction of the corresponding function satisfying dissipation inequality of CW systems. The result is used to analyze the stability of a second order system with hysteretic friction which is described by a Dahl model.
\end{abstract}


\end{frontmatter}


\section{Introduction}
Hysteresis is a common nonlinear phenomena that is present in diverse physical systems, such as piezo-actuator, ferromagnetic material and mechanical systems. From the perspective of input-output behavior, the hysteretic phenomena can be characterized into counterclockwise (CCW) input-output (I/O) dynamics \cite{ANGELI2006}, clockwise (CW) I/O dynamics \cite{PADTHE2005}, or even more complex I/O map (such as, butterfly map \cite{BERTOTTI2006}). For example, backlash operator generates CCW I/O dynamics; elastic-plastic operator generates CW I/O dynamics and Preisach operator can have either CCW or CW I/O dynamics depending on the weight of the hysterons which are used in the Preisach model \cite{BROKATE1996, MACKI1993, LOGEMANN2003}.

In the recent work by Angeli \cite{ANGELI2006}, the counterclockwise (CCW) I/O dynamics of a single-input single-output system is characterized by the following inequality
\begin{equation}\label{CCWineq}
\liminf_{T\rightarrow \infty}\int^T_0{\dot{y}(t)u(t)dt > -\infty},
\end{equation}
where $u$ is the input signal and $y$ is the corresponding output signal. It is assumed that $u\in U$ where $U$ is the set of input signals for which $y$ exists and is well defined for all positive time. Compare with the classical definition of passivity \cite{WILLEMS1972}, it can be interpreted as the system is passive from the input $u$ to the time derivative of the corresponding output $y$. In particular, \rfb{CCWineq} holds if there exists a function $H: \rline^2\rightarrow \rline_+$ such that
\begin{equation} \label{ccwderiv}
\frac{\dd H(y(t),u(t))}{\dd t} \leq \dot{y}(t)u(t).
\end{equation}
Indeed, integrating \rfb{ccwderiv} from $0$ to $\infty$ we obtain \rfb{CCWineq}.

Correspondingly, clockwise (CW) I/O dynamics can be described by the following dissipation inequality
\begin{equation}\label{CWineq}
\liminf_{T\rightarrow \infty}\int^T_0{\dot{u}(t)y(t)dt > -\infty}.
\end{equation}
The notions of counterclockwise (CCW) I/O and  clockwise (CW) I/O are also discussed in \cite{OH2005}.

In our previous results in \cite{JAYAWARDHANAAUTO2012}, we show that for a certain class of Duhem hysteresis operator $\Phi:u \mapsto \Phi(u,y_0):=y$, we can construct a function $H_{\circlearrowleft} : \rline^2 \rightarrow \rline_+$ which satisfies
\begin{equation}\label{passivehysteresis}
\frac{\dd H_{\circlearrowleft}(y(t),u(t))}{\dd t} \leq \dot{y}(t)u(t).
\end{equation}
This inequality immediately implies that such Duhem hysteresis operator is dissipative with respect to the supply rate $\dot{y}(t)u(t)$ and has CCW input-output dynamics. The symbol $\circlearrowleft$ in $H_{\circlearrowleft}$ indicates the counterclockwise behavior of $\Phi$.

In this paper, as a dual extension to \cite{JAYAWARDHANAAUTO2012}, we focus on the clockwise (CW) hysteresis operator where the supply rate is given by $\dot{u}y$ which is dual to the supply rate $u\dot{y}$ considered in \cite{JAYAWARDHANAAUTO2012}. This is motivated by the friction induced hysteresis phenomenon in the mechanical system which has CW I/O behavior from the input relative displacement to the output friction force. One may intuitively consider to reverse the input-output relation of the CW hysteresis operator for getting the CCW I/O behavior in the reverse I/O setting. However, this consideration has two drawbacks: 1). the reverse input-output pair may not be physically realizable (this is related to the causality problem in the port-based modeling, such as, the bond graph modeling framework \cite{PETER2004}); 2). the operator itself may not be invertible (for example, if the output of the hysteresis operator can be saturated).

In Theorem \ref{thm1}, we provide sufficient conditions on the underlying functions $f_1$ and $f_2$ of the Duhem operator, such that it has CW I/O dynamics. Roughly speaking, the functions $f_1$ and $f_2$ (as defined later in Section 2) determine two possible different directions $(y,u)$ depending on whether the input $u$ is increasing or decreasing. By evaluating these two functions on two disjoint domains (which are separated by an anhysteresis curve), we can determine whether it has CW I/O dynamics using Theorem \ref{thm1}. This is shown by constructing a function $H_{\circlearrowright} : \rline^2 \rightarrow \rline_+$ such that the following inequality
\begin{equation}\label{passivehysteresis_cw}
\frac{\dd H_{\circlearrowright}(y(t),u(t))}{\dd t} \leq y(t)\dot{u}(t).
\end{equation}
holds. The function $H_{\circlearrowright}$ can also be related to the concept of available storage function from \cite{WILLEMS1972} where, instead of using the standard supply rate $yu$, we use the CW supply rate $y\dot{u}$ as shown in Proposition 1 in this paper.

The dissipativity property \rfb{passivehysteresis_cw} can be further used in the stability analysis of the systems with CW hysteresis, such as, a second-order mechanical system with hysteretic friction as discussed in Section 4.2. As an illustrative example on the application of \rfb{passivehysteresis_cw}, let us consider a mechanical system described by
\begin{equation*}
\left.\begin{array}{rl}
 m\ddot{x} & = F- F_{{ \rm friction}},\\
 F_{{ \rm friction}}& =\Phi(x,y_0), \end{array}\right.
\end{equation*}
with the hysteresis operator $\Phi$ satisfying the Dahl model as follows
\begin{equation*}
\dot{F}_{\rm friction}=\rho \left (1-\frac{F_{\rm friction}}{F_C}\right )\max\{0, \dot{x}\}+\rho\left(1+\frac{F_{\rm friction}}{F_C}\right)\min\{0,\dot{x}\},
\end{equation*}
where $m$ refers to the mass, $x$ refers to the displacement, $F$ is the applied force, $\rho>0$ describes the stiffness constant, $F_C>0$ represents the Coulomb friction constant and $y_0$ is the initial condition of the Dahl model (see, for example, \cite{OH2005}). By taking $x_1=x, x_2=\dot{x}$ and $x_3=F_{friction}$ as the state variables, we can rewrite this hysteretic system into state-space form as follows
\begin{equation*}
\left.\begin{array}{rl}
\dot{x}_1&=x_2,\\
\dot{x}_2&=\frac{F}{m}-\frac{x_3}{m},\\
\dot{x}_3&=\rho \left (1-\frac{x_3}{F_C}\right)\max\{0,x_{2}\}+\rho\left (1+\frac{x_3}{F_C}\right)\min\{0,x_{2}\}.\end{array} \right.
\end{equation*}
In Section 4.1, we obtain the function $H_{\circlearrowright}$ satisfying \rfb{passivehysteresis_cw} explicitly and it is parameterized by $\rho$ and $F_C$. Using $V(x_1,x_2,x_3)=\frac{1}{2}mx_2^2+H_{\circlearrowright}(x_3,x_1)$ as a Lyapunov function we have
 \begin{align*}
 \dot{V}&=m\dot{x}_2x_2+\frac{\dd H_{\circlearrowright}(x_3,x_1)}{\dd t}\\
        &=-x_3x_2+Fx_2+\frac{\dd H_{\circlearrowright}(x_3,x_1)}{\dd t}\\
        &\leq Fx_2.
 \end{align*}
 This inequality establishes that the closed loop system is passive from the applied force $F$ to the velocity $x_2$. Thus a simple propositional feedback $F=-dx_2$, where $d>0$, can guarantee the asymptotic convergence of the velocity $x_2$ to zero without having to know precisely the parameters $\rho$ and $F_C$.

\section{Duhem operator and clockwise hysteresis operators}\label{hysmodel}
Denote $C^1(\rline_+)$ the space of continuously differentiable functions $f:\rline_+\to \rline$ and $AC(\rline_+)$ the space of absolutely continuous functions $f:\rline_+ \rightarrow \rline$. Define $\frac{\dd z(t)}{\dd t} := \lim_{h\searrow 0^+}\frac{z(t+h)-z(t)}{h}$.

The Duhem operator $\Phi:AC(\rline_+)\times \rline \to AC(\rline_+), (u,y_0)\mapsto \Phi(u,y_0)=:y$ is described by \cite{MACKI1993, OH2005, VISINTIN1994}
\begin{equation}\label{babuskamodel}
\dot y(t) = f_1(y(t),u(t))\dot u_+(t) + f_2(y(t),u(t))\dot u_-(t),\ y(0)=y_0,
\end{equation}
where $\dot u_+(t):=\max\{0,\dot u(t)\}$, $\dot u_-(t):=\min\{0,\dot u(t)\}$. The functions $f_1$ and $f_2$ are assumed to be $C^1$.

The existence of solutions to \rfb{babuskamodel} has been reviewed in \cite{MACKI1993}. In particular, if for every $\xi \in \rline $, $f_1$ and $f_2$ satisfy
\begin{align}\label{exis_cond}
(\sigma_1 - \sigma_2)[f_1(\sigma_1,\xi) - f_1(\sigma_2,\xi)] & \leq \lambda_1(\xi)(\sigma_1 - \sigma_2)^2,\\\nonumber
(\sigma_1 - \sigma_2)[f_2(\sigma_1,\xi) - f_2(\sigma_2,\xi)] & \geq -\lambda_2(\xi)(\sigma_1 - \sigma_2)^2,
\end{align}
for all $\sigma_1$, $\sigma_2 \in \rline$, where $\lambda_1$ and $\lambda_2$ are nonnegative, then the solution to \rfb{babuskamodel} exist and $\Phi$ maps $AC(\rline_+) \times \rline \rightarrow AC(\rline_+)$. We will assume throughout the paper that the solution to \rfb{babuskamodel} exists for all $u\in AC(\rline_+)$ and $y_0 \in \rline$.

As a dual definition to counterclockwise (CCW) I/O behavior \cite{ANGELI2006}, we define the clockwise (CW) I/O dynamics as follows
\begin{definition}\label{de-CW}
An operator $Q$ is {\it clockwise} (CW) if for every $u \in U$ with the corresponding output map $y:=Qu$, where $U$ is the space of input signals such that $y$ is well-defined for all positive time, the following inequality holds
\begin{equation}\label{defCW}
\liminf_{T\rightarrow \infty}\int^T_0{y(t)\dot{u}(t)\dd t > -\infty}.
\end{equation}
\end{definition}

\vspace{0.2cm}
For the Duhem operator $\Phi$, inequality \rfb{defCW} holds if there exists a function $H_{\circlearrowright}: \rline^2 \rightarrow \rline_+$ such that for every $u\in AC(\rline_+)$ and $y_0 \in \rline$, the inequality
\begin{equation}\label{CWcon}
\frac{\dd H_{\circlearrowright}(y(t),u(t))}{\dd t} \leq y(t)\dot{u}(t),
\end{equation}
holds for all $t$ where $y := \Phi(u,y_0)$.

In the following subsections, we describe several well-known hysteresis operators which generate clockwise I/O dynamics and we recast these operators into the Duhem operator as in \rfb{babuskamodel}.

\subsection{Dahl model}\label{dahlsec}
The Dahl model \cite{DAHL1976,PADTHE2008} is commonly used in mechanical systems, which represents the friction force with respect to the relative displacement between two surfaces in contact. The general representation of the Dahl model is given by
\begin{equation}
\dot{y}(t)=\rho\left |1-\frac{y(t)}{F_c}\textrm{sgn} (\dot{u}(t))\right |^{r}\textrm{sgn} \left (1-\frac{y(t)}{F_c}\textrm{sgn}(\dot{u}(t))\right)\dot{u}(t),
\end{equation}
where $y$ denotes the friction force, $u$ denotes the relative displacement, $F_c>0$ denotes the Coulomb friction force, $\rho >0$ denotes the rest stiffness and $r\geq 1$ is a parameter that determines the shape of the hysteresis loops.

The Dahl model can be described by the Duhem hysteresis operator \rfb{babuskamodel} with
\begin{equation}
f_1(\sigma,\xi) = \rho \left | 1-\frac{\sigma}{F_c}\right |^{r} \textrm{sgn}\left (1-\frac{\sigma}{F_c}\right ),
\end{equation}
\begin{equation}
f_2(\sigma,\xi) = \rho \left | 1+\frac{\sigma}{F_c}\right |^{r} \textrm{sgn}\left (1+\frac{\sigma}{F_c}\right ).
\end{equation}
In Figure \ref{dahl}, we illustrate the behavior of the Dahl model where $F_c=0.75$, $\rho=1.5$ and $r=3$.
\begin{figure}[h]
\centering \psfrag{D}{$y(t)$}\psfrag{u}{$u(t)$}\psfrag{t}{$t$}
 \subfloat[] {{\includegraphics[width=2.4in]{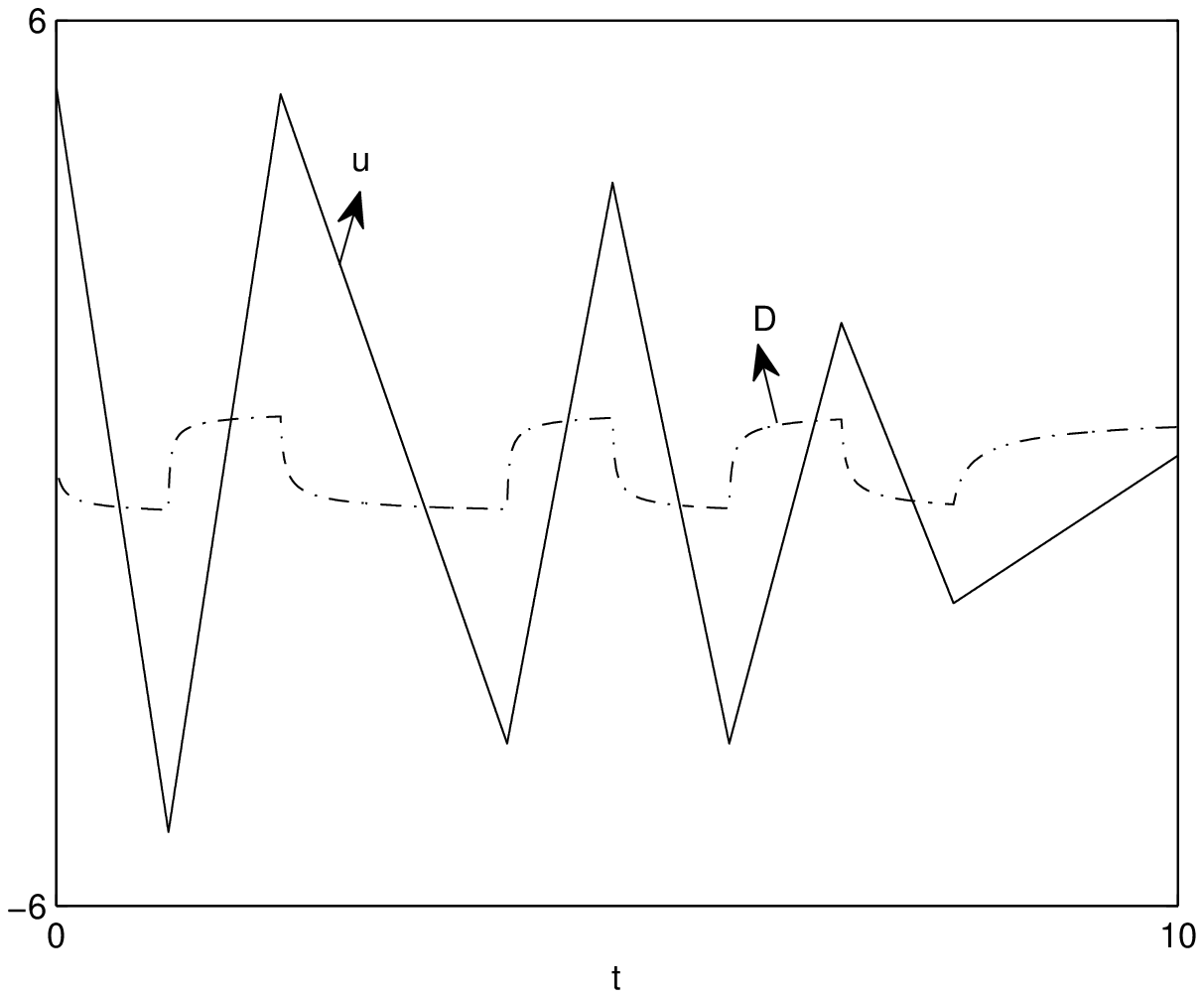}}}
 \subfloat[] {{\includegraphics[width=2.4in]{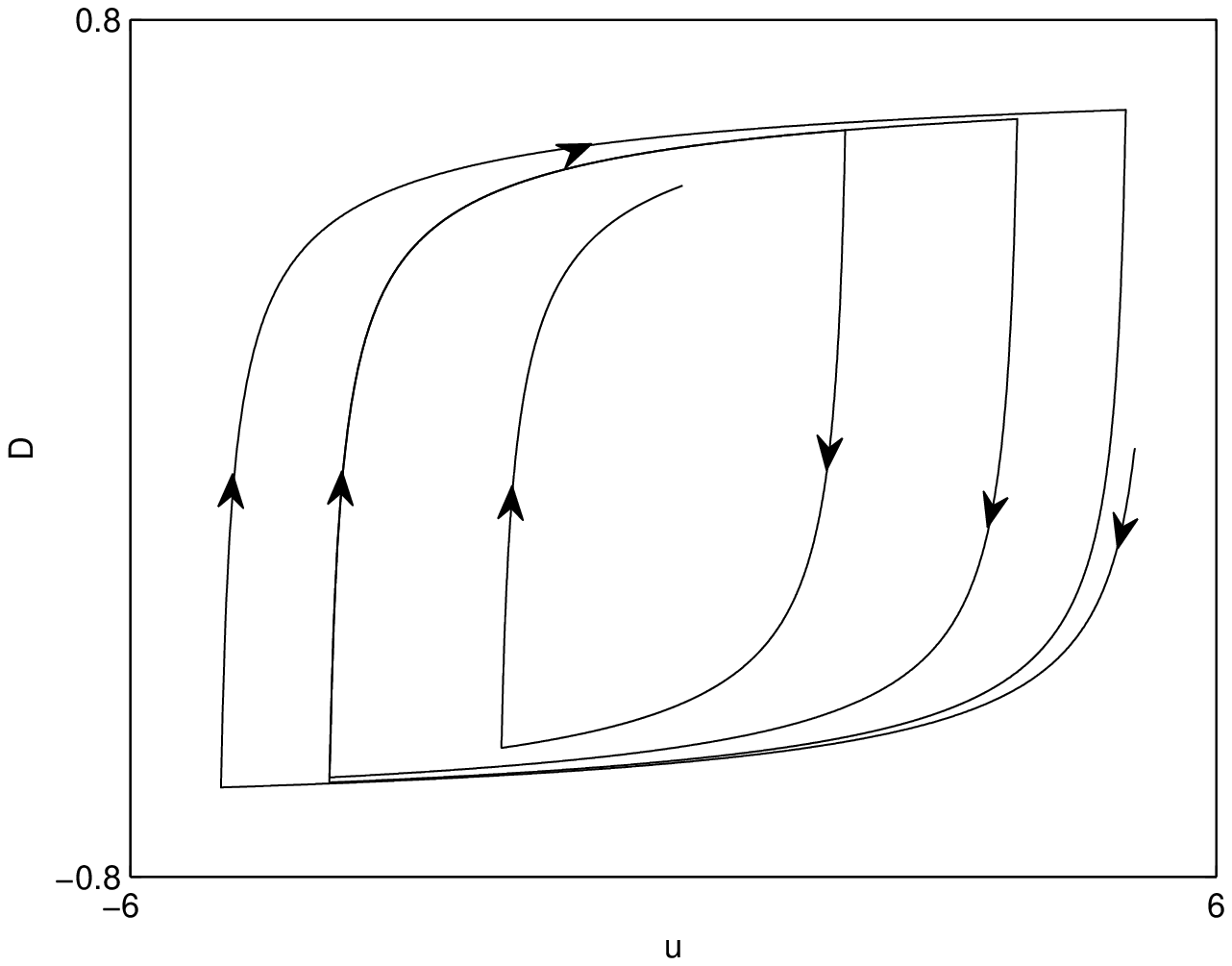}}}
\caption{The input-output dynamcis of the Dahl model with $F_c=0.75$, $\rho=1.5$ and $r=3$.}
\label{dahl} \vspace{0.2cm}
\end{figure}
\subsection{Bouc-Wen model}
The Bouc-Wen model \cite{SAIN1997, WEN1976} is commonly used to model the elastic stress-strain relationships in structures. Moreover, it is also used to represent the magnetorheological behavior in the MR damper \cite{DYKE1996}. The general representation of the Bouc-Wen model is given by
\begin{equation*}
\dot{y}(t)=\alpha \dot{u}(t)-\beta \dot{u}(t)|y(t) |^n-\gamma |\dot{u}(t)|y(t)|y(t)|^{n-1},
\end{equation*}
where $u$ denotes the displacement, $y$ denotes the elastic strain, $n\geq 1$ and $\beta, \zeta$ are the parameters determine the shape of the hysteresis curve.

The Bouc-Wen model can be described by the Duhem hysteresis operator \rfb{babuskamodel} with
\begin{equation}
f_1(\sigma,\xi) = \alpha-\beta |\sigma |^n-\zeta \sigma|\sigma|^{n-1},
\end{equation}
\begin{equation}
f_2(\sigma,\xi) = \alpha-\beta |\sigma |^n+\zeta \sigma|\sigma|^{n-1}.
\end{equation}
In Figure \ref{bouc}, we illustrate the behavior of the Bouc-Wen model where $\alpha=1$, $\beta=1$, $\zeta=1$ and $n=3$.
\begin{figure}[h]
\centering \psfrag{y}{$y(t)$}\psfrag{u}{$u(t)$}\psfrag{t}{$t$}
 \subfloat[] {{\includegraphics[width=2.4in]{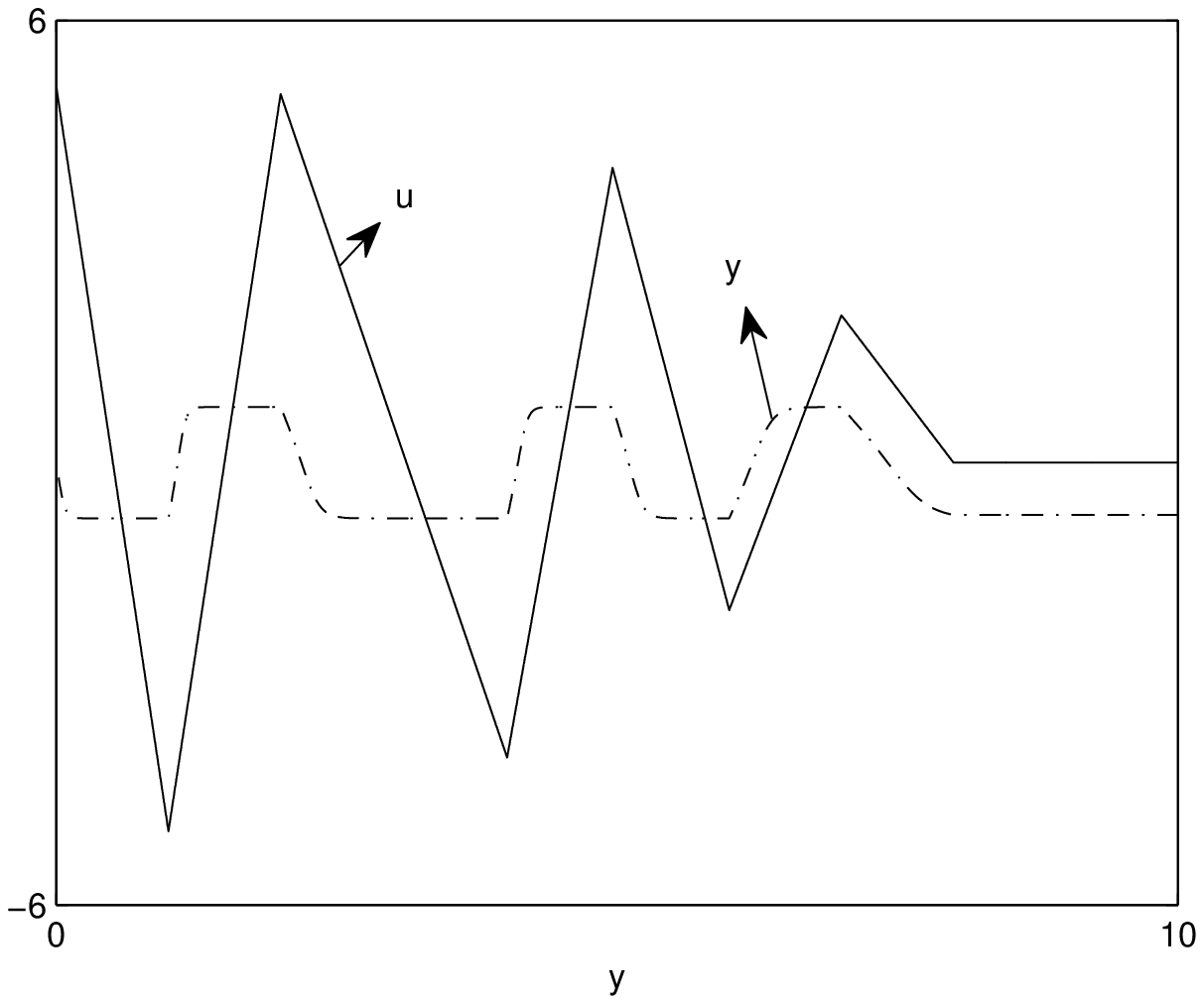}}}
 \subfloat[] {{\includegraphics[width=2.4in]{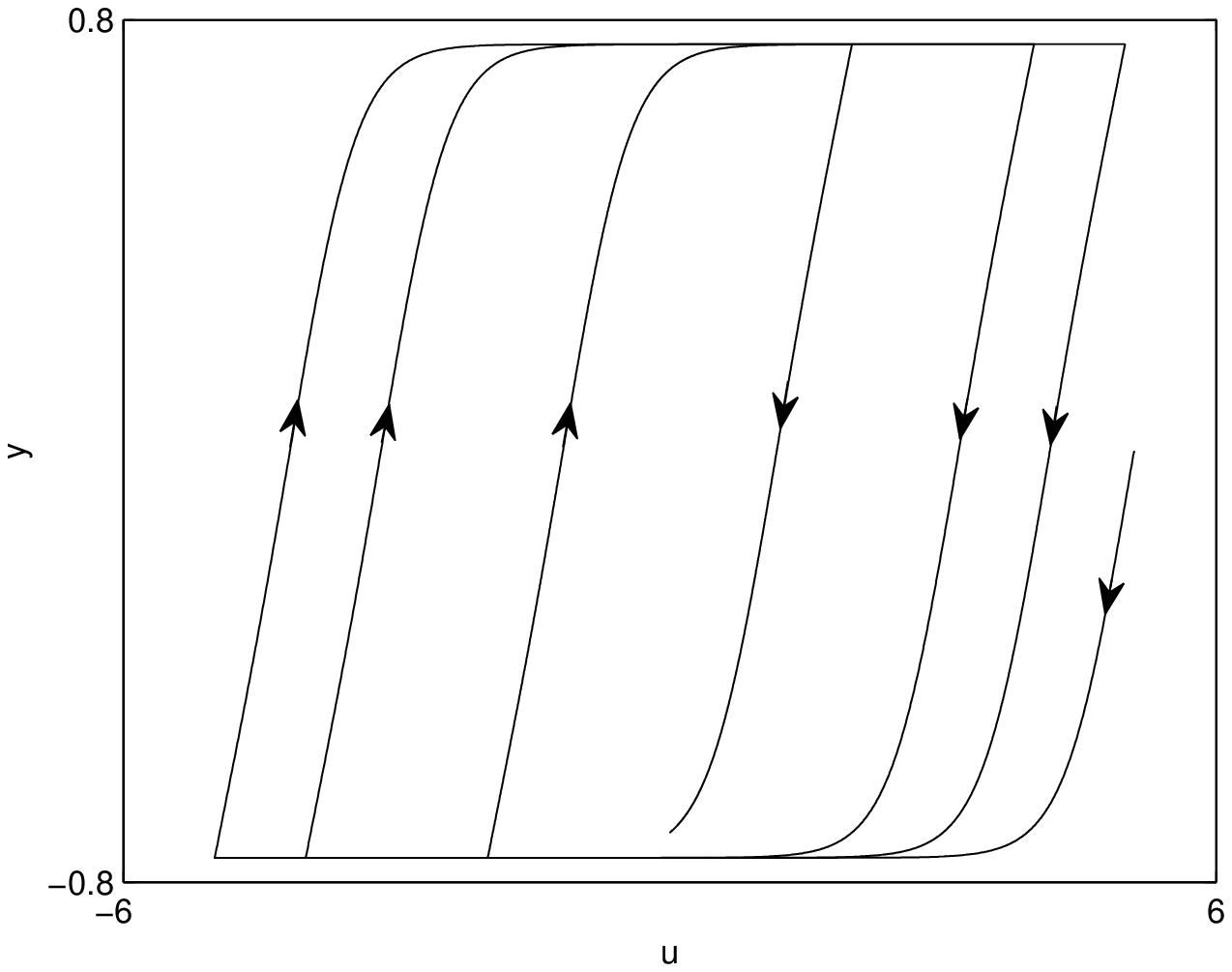}}}
\caption{The input-output dynamcis of the Bouc-Wen model with $\alpha=1$, $\beta=1$, $\zeta=1$ and $n=3$.}
\label{bouc} \vspace{0.2cm}
\end{figure}
\section{Main result}
Before stating our main contribution, we need to introduce three functions in the following subsections: an anhysteresis function $f_{an}$, a traversing function $\omega_\Phi$ and an intersecting function $\Lambda$; these functions will play an important role in the characterization of dissipativity and in the construction of the storage function. These three functions are defined based on the knowledge of $f_1$ and $f_2$. Generally speaking, the anhysteresis function $f_{an}$ defines the curve where $f_1=f_2$, the function $\omega_\Phi$ describes the trajectory of $\Phi$ when a monotone increasing $u$ or a monotone decreasing $u$ is applied from a given point in the hysteresis phase plot, and the intersecting function $\Lambda$ defines the intersection of the anhysteresis function $f_{an}$ and function $\omega_\Phi$ from a given point. The anhysteresis function $f_{an}$ and the traversing function $\omega_\Phi$ have the same definitions as given in our previous results in \cite{JAYAWARDHANACDC2011}.

\subsection{Anhysteresis function}
In order to define the anhysteresis function, we rewrite $f_1$ and $f_2$ as follows
\begin{equation}\label{duhemf1f2}
\left. \begin{array}{ll} f_1(y(t),u(t)) & = F(y(t),u(t))+G(y(t),u(t)), \\
 f_2(y(t),u(t)) &= -F(y(t),u(t))+ G(y(t),u(t)), \end{array} \right\}
\end{equation}
where $F,G: \rline^2 \rightarrow \rline$. We assume that the implicit function $F(\sigma,\xi)=0$ can be represented by an explicit function $\sigma=f_{an}(\xi)$ or $\xi=g_{an}(\sigma)$. Such function $f_{an}$ (or $g_{an}$) is called an {\it anhysteresis function} and the corresponding graph $\{(\xi,f_{an}(\xi))|\xi\in\rline\}$ is called an {\it anhysteresis curve}. Using $f_{an}$, it can be checked that $f_1(f_{an}(\xi),\xi)=f_2(f_{an}(\xi),\xi)$ holds. Note also that the functions $F$ and $G$ in \rfb{duhemf1f2} are defined by
\begin{equation*}
F = \frac{f_1-f_2}{2} \qquad G = \frac{f_1+f_2}{2}.
\end{equation*}

\subsection{Traversing function $\omega_{\Phi}$}
 For every given point $(\sigma,\xi)\in \R^2$ in the hysteresis phase plot, let
$\omega_{\Phi,1}(\cdot,\sigma,\xi):[\xi,\infty)\to \R$ be the solution $x$ of
\begin{equation*}
x(\tau) - x(\xi) = \int^\tau_{\xi}{f_1(x(\lambda),\lambda)\
\dd\lambda}\ \quad x(\xi)=\sigma \quad \forall \tau\in[\xi,\infty),
\end{equation*}
and let $\omega_{\Phi,2}(\cdot,\sigma,\xi):(-\infty,\xi]\to\rline$ be
the solution $x$ of
\begin{equation*}
x(\tau) - x(\xi) = \int_{\xi}^\tau{f_2(x(\lambda),\lambda)\
\dd\lambda}\ \quad  x(\xi)=\sigma \quad \forall \tau\in (-\infty,\xi].
\end{equation*}
Using the above definitions, for every point $(\sigma,\xi)\in \R^2$ in the hysteresis phase plot, the {\it traversing function}
$\omega_\Phi(\cdot,\sigma,\xi):\R\to \R$ is defined by the
concatenation of $\omega_{\Phi,2}(\cdot,\sigma,\xi)$ and
$\omega_{\Phi,1}(\cdot,\sigma,\xi)$:
\begin{equation}\label{wcurve}
\omega_\Phi(\tau,\sigma,\xi) = \left\{ \begin{array}{ll}
\omega_{\Phi,2}(\tau,\sigma,\xi) & \forall \tau \in (-\infty,\xi), \\
\omega_{\Phi,1}(\tau,\sigma,\xi) & \forall \tau \in [\xi,\infty).
\end{array} \right.
\end{equation}
We remark that the function $\omega_\Phi(\cdot,\sigma,\xi)$ defines the (unique) hysteresis curve where the curve $\{(\tau,\omega_\Phi(\tau,\sigma,\xi))\,|\,\tau\in (-\infty, \xi]\}$ is obtained by applying a monotone decreasing $u$ to $\Phi(\cdot,\sigma)$ with $u(0)=\xi$, $\lim_{t\to\infty}u(t)=-\infty$ and, similarly, the curve $\{(\tau,\omega_\Phi(\tau,\sigma,\xi))\,|\,\tau\in [\xi,\infty)\}$ is obtained by introducing a monotone increasing $u$ to $\Phi(\cdot,\sigma)$ with $u(0)=\xi$ and $\lim_{t\to\infty}u(t)=\infty$.\\
\subsection{Intersecting function $\Lambda$}
The intersecting function $\Lambda$ describes the intersection between the anhysteresis curve $f_{an}$ and the curve $\omega_\Phi$. The function $\Lambda:\rline^2\to \rline$ is an {\it intersecting function} (corresponding to $\omega_\Phi$ and $f_{an}$) if: i) $\omega_\Phi(\Lambda(\sigma,\xi),\sigma,\xi)=f_{an}(\Lambda(\sigma,\xi))$ for all $(\sigma,\xi)\in\rline^2$ and; ii) $\Lambda(\sigma,\xi)\leq \xi$ whenever $\sigma \geq f_{an}(\xi)$ and $\Lambda(\sigma,\xi)>\xi$ otherwise. This implies that the two functions $\omega_\Phi(\cdot,\sigma,\xi)$ and $f_{an}(\cdot)$ intersect at a unique point larger or smaller than $\xi$ depending on the sign of $\sigma - f_{an}(\xi)$. In our main result, we also need that $\frac{\dd \Lambda(y(t),u(t))}{\dd t}$ exists for every solutions $(y,u)$ of \rfb{babuskamodel}.

In the following lemma we give sufficient conditions for the existence of such intersecting function $\Lambda$.
\begin{lemma}\label{lemma_ass_intersect}
Assume that $f_1$ and $f_2$ in \rfb{duhemf1f2} be such that $f_1$, $f_2$ are $C^1$. Moreover, assume that $f_{an}$ is strictly increasing and there exists a positive real constant $\epsilon >0$ such that for all $(\sigma,\xi)\in \rline^2$ the following inequality holds
\begin{eqnarray}
\label{eq_ass_Lambda_1}
f_1(\sigma,\xi) >  \frac{\dd f_{an}(\xi)}{\dd \xi}+\epsilon & \textrm{ whenever }& \sigma > f_{an}(\xi)\ ,\\
\label{eq_ass_Lambda_2}
f_2(\sigma,\xi) >  \frac{\dd f_{an}(\xi)}{\dd \xi}+\epsilon & \textrm{ whenever }& \sigma < f_{an}(\xi)\ .
\end{eqnarray}
Then there exists an intersecting function $\Lambda \in C^1(\rline^2,\rline)$ such that
\begin{description}
\item[\bf{(1)}] $\Lambda(\sigma,\xi)\leq \xi$ whenever $\sigma \geq f_{an}(\xi)$ and $\Lambda(\sigma,\xi)>\xi$ otherwise.
\item[\bf{(2)}]
$\displaystyle
\omega_\Phi(\Lambda(\sigma,\xi),\sigma,\xi)=f_{an}(\Lambda(\sigma,\xi)).
$ \hfill \refstepcounter{equation}\label{Omega}($\theequation$)
\item[\bf{(3)}] Moreover, for all $u\in C^1$, $y:=\Phi(u,y_0)$, we have that $\frac{\dd}{\dd t} \Lambda(y(t),u(t))$ exists.
\end{description}
\end{lemma}
The proof of Lemma \ref{lemma_ass_intersect} is given in the \ref{prooflemmaintersect}.

\begin{example}\label{example1}
{\rm In order to illustrate these functions, let us consider the Duhem operator $\Phi$ with $f_1(\sigma,\xi) =e^{0.5(-1.2\sigma+\xi)}+0.83$ and $f_2(\sigma,\xi)=e^{0.5(1.2\sigma-\xi)}+0.83$ as shown in Figure \ref{storagefig}. It can be checked that the anhysteresis function of the operator is $f_{an}(\xi)=0.83\xi$ and the functions $f_1$ and $f_2$ satisfy the hypotheses in Lemma \ref{lemma_ass_intersect}. With a reference to Figure \ref{storagefig}, let the current state of $\Phi$ be given by $(y(t),u(t))$. In this figure, the traversing function $\omega_{\phi}(\cdot,y(t),u(t))$ is depicted by the dashed-line and the anhysteresis function $f_{an}$ is shown by the thick solid-line. The point $(y(t),u(t))$ is located above the anhysteresis curve, i.e., $y(t)>f_{an}(u(t))$. It can be seen from the figure that the intersecting point $\Lambda(y(t),u(t))$ (which is shown by the solid circle) is less than $u(t)$, i.e., $\Lambda(y(t),u(t))\leq u(t)$. This shows that the property }{\bf{(1)}} {\rm in Lemma \ref{lemma_ass_intersect} holds.}
\end{example}
\begin{figure}
\centering \psfrag{Y}{$y(t)$} \psfrag{U}{$u(t)$} \psfrag{w}{$\omega_{\Phi}$}\psfrag{fan}{$f_{an}$}
\psfrag{y}{$y$}\psfrag{u}{$u$ }\psfrag{L}{$\Lambda(y(t),u(t))$}
  \includegraphics[height=2.5in]{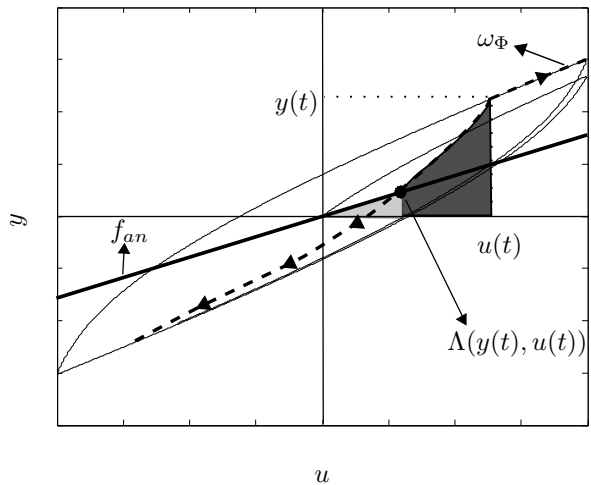}\\
  \caption{Illustration of a Duhem operator with $f_1(\sigma,\xi) =e^{0.5(-1.2\sigma+\xi)}+0.83$ and $f_2(\sigma,\xi)=e^{0.5(1.2\sigma-\xi)}+0.83$ for all $(\sigma,\xi)\in \rline^2$. The anhysteresis curve $f_{an}(\xi)=0.83\xi$ is shown by the thick solid-line. If the current state be given by $(y(t),u(t))$, the traversing function $\omega_{\phi}(\cdot,y(t),u(t))$ is depicted by the dashed-line and the intersecting point $\Lambda(y(t),u(t))$ is shown by the solid circle.}\label{storagefig}
\end{figure}

\subsection{Duhem operator with clockwise hysteresis}
Based on the three functions $\omega_{\Phi}$, $f_{an}$ and $\Lambda$, we define $H_{\circlearrowright}: \rline^2\rightarrow\rline_+$ as follows
\begin{equation}\label{storage_cw}
H_{\circlearrowright}(\sigma,\xi)=\int_{0}^{\Lambda(\sigma, \xi)}{f_{an}(\tau) \dd \tau}-\int_{\xi}^{\Lambda(\sigma,\xi)}{\omega_\Phi(\tau,\sigma,\xi) \dd \tau}.
\end{equation}

\begin{theorem}\label{thm1}
Consider the Duhem hysteresis operator $\Phi$ defined in \rfb{babuskamodel} and \rfb{duhemf1f2} with $C^1$ functions $F, G:\rline^2 \to \rline$ and with the traversing function $\omega_{\Phi}$ and the anhysteresis function $f_{an}$. Suppose that there exists an intersecting function $\Lambda$ (e.g. the hypotheses in Lemma \ref{lemma_ass_intersect} hold). Let the following condition holds for all $(\sigma,\xi)$ in $\rline^2$
\begin{description}
\item[{\bf (A)}] $F(\sigma,\xi)\geq 0$ whenever $\sigma \leq f_{an}(\xi)$, and $F(\sigma,\xi) < 0$ otherwise.
\end{description}
Then for every $u\in AC(\rline_+)$ and for every $y_0\in\rline$, the function $t\rightarrow H_{\circlearrowright}(y(t),u(t))$ with $H_{\circlearrowright}$ as in \rfb{storage_cw} and $y:=\Phi(u,y_0)$, is right differentiable and satisfies \rfb{passivehysteresis_cw}. Moreover, if the anhysteresis function $f_{an}$ satisfies $f_{an}(0)= 0$, then $H_{\circlearrowright}\geq 0$ and the Duhem operator is clockwise (CW).
\end{theorem}
The proof of Theorem \ref{thm1} is given in the \ref{proofofthm1}.

\begin{remark}
In addition to the result in Theorem \ref{thm1}, if $f_1$ and $f_2$ satisfy the hypotheses given in Theorem \ref{thm1}, then for every $u\in AC(\rline_+)$ and $y_0 \in \rline$, the function $t \rightarrow H_{\circlearrowright}(y(t),u(t))$ with $H_{\circlearrowright}$ as in \rfb{storage_cw} is left-differentiable and satisfies
\begin{equation*}
\lim_{h\nearrow 0^-}\frac{H_{\circlearrowright}(y(t+h),u(t+h))-H_{\circlearrowright}(y(t),u(t))}{h}\leq y(t)\dot{u}(t).
\end{equation*}
The proof of this claim follows a similar line as that of Theorem \ref{thm1}.
\end{remark}

In order to depict the storage function $H_{\circlearrowright}$ that is constructed in Theorem \ref{thm1}, we recall again the example of the Duhem operator $\Phi$ in Example \ref{example1} where $f_1=e^{0.5(-1.2y+u)}+0.83$ and $f_2=e^{0.5(1.2y-u)}+0.83$, and it is shown in Figure \ref{storagefig}. Based on the functions $f_{an}$ and $\omega_{\Phi}(\cdot,y(t),u(t))$ as shown in Figure \ref{storagefig} and following the construction of the storage function $H_{\circlearrowright}$ as in \rfb{storage_cw}, the first component on the RHS of \rfb{storage_cw} corresponds to the light grey area in Figure \ref{storagefig}. Correspondingly, the second component on the RHS of \rfb{storage_cw} refers to the dark grey area in Figure \ref{storagefig}. The summation of these two areas gives the storage function $H_{\circlearrowright}$ for a given state $(y(t),u(t))$ satisfying \rfb{passivehysteresis_cw} according to Theorem \ref{thm1}.

The principle of the construction of $H_{\circlearrowright}$ in \rfb{storage_cw} can be described in words as follows. From a given state $(y(0),u(0))$, let us define the trajectory that crosses the anhysteresis curve at a given time $T$ by applying either a monotonically increasing input signal $u(t)=u(0)+t$ or a monotonically decreasing input signal $u(t)=u(0)-t$. Denote this trajectory by $y$ and the intersecting point by $(y(T),u(T))$. Then the storage function $H_{\circlearrowright}$ is given by the integral of the anhysteresis function from $0$ to $u(T)$ minus the integral of $y$ from $0$ to $T$.

\begin{proposition}
Consider the Duhem operator $\Phi$ satisfying the hypotheses in Theorem \ref{thm1}. Moreover, we assume that the anhysteresis function $f_{an}$ satisfies $f_{an}(\xi)=0$ for all $\xi \in \rline$. Then for every $y_0,u_0\in\rline$, the function $H_{\circlearrowright}$ as in \rfb{storage_cw} satisfies
\begin{equation*}
H_{\circlearrowright}(y_0,u_0)=\sup\limits_{\substack{u\in AC(\rline_+)\\u(0)=u_0}}-\int_0^T{y(\tau)\dot{u}(\tau)\dd \tau},
\end{equation*}
 where $y:=\Phi(u,y_0)$. In other words, $H_{\circlearrowright}$ defines the available storage function (as discussed in \cite{WILLEMS1972}) where the supply rate is given by $y\dot{u}$ (instead of $yu$ as in \cite{WILLEMS1972}).
\end{proposition}
\begin{proof}
%
As given in the first part of the proof of Theorem \ref{thm1}, we have
\begin{equation}\label{Hdot2}
\frac{\dd}{\dd t}H_{\circlearrowright}(y(t),u(t)) =\dot{u}(t) y(t) - \int_{u(t)}^{u^*(t)}{\frac{\dd}{\dd t} \omega_\Phi(\tau,y(t),u(t)) \dd\tau}.
\end{equation}
Integrating \rfb{Hdot2} from $t=0$ to $T$, we obtain
\begin{equation*}
H_{\circlearrowright}(T)-H_{\circlearrowright}(0) = \int_{0}^{T}{y(\tau)\dot{u}(\tau) \dd \tau} - \int_{0}^{T}{\int_{u(t)}^{u^*}{\frac{\dd}{\dd t}\omega_{\Phi}(\tau,y(t),u(t)\dd \tau \dd t,}}
\end{equation*}
where $u^*=\Lambda(y(t),u(t))$ and we have used the shorthand notation of $H_\circlearrowright(t):=H_\circlearrowright(y(t),u(t))$.

By rearranging the terms in this equation, we arrive at
\begin{equation}\label{avai_div1}
-\int_{0}^{T}{y(\tau)\dot{u}(\tau) \dd \tau}=H_{\circlearrowright}(0)-H_{\circlearrowright}(T)- \int_{0}^{T}{\int_{u(t)}^{u^*}{\frac{\dd}{\dd t}\omega_{\Phi}(\tau,y(t),u(t)\dd \tau \dd t}}.
\end{equation}
The supremum of the LHS of \rfb{avai_div1} over all possible $u$ and $T$ defines the available storage function where the supply rate is $y\dot u$. Note that this supply rate is a particular class of the general supply rate as studied in \cite{TRENTELMANN1997, WILLEMS1972}. Since the last two terms on the RHS of \rfb{avai_div1} is non-positive, we will show that we can define $u$ and $T$ such that these two terms are equal to zero, and thus the supremum of the LHS of \rfb{avai_div1} is equal to $H_{\circlearrowright}(y(0),u(0))$, which is equivalent to $H_{\circlearrowright}(y_0,u_0)$, i.e., $H_\circlearrowright$ is the available storage function.

From a given initial condition $(y_0,u_0)$, let us introduce an input signal $u(t)=u_0(T-t)+t\Lambda(y_0,u_0)$ for all $t\in[0, T]$ and $u(t)=\Lambda(y_0,u_0)$ otherwise. This means that we have an input signal $u$ which starts from $u_0$, ends at $\Lambda(y_0,u_0)$ at $t=T$ and remains there for all $t>T$. Together with the corresponding signal $y=\Phi(u,y_0)$, we have $\Lambda(y(t),u(t))=\Lambda(y_0,u_0)$ for all $t$, i.e. the intersecting point is always the same. Indeed, this follows from the fact that $\Lambda(y(t),u(t))$ remains the same along the trajectories that converge to the intersection point $(\omega_{\Phi}(u^*,y_0,u_0),u^*)$ where $u^*=\Lambda(y_0,u_0)$.

Following the same arguments as in the proof of Theorem \ref{thm1} (c.f., the arguments that lead to Eq. \rfb{proofeq}), this input signal ensures that the last term on the RHS of \rfb{avai_div1} is equal to zero. Since $u(T)=\Lambda(y_0,u_0)$ for all $t>T$, we also have that $H_\circlearrowright(y(t),u(t))=0$ for all $t>T$, i.e. the second term on the RHS of \rfb{avai_div1} is zero using such an input signal. Hence $H_{\circlearrowright}$ as in \rfb{storage_cw} is an available storage function.
\end{proof}

The results given in Theorem \ref{thm1} can be slightly generalized in order to incorporate the case when the Duhem hysteresis operator $\Phi$ has saturated output. Consider the set $D\subset \rline^2$ which contains all relations of $\Phi$, i.e., $\calR_{y_0,u}:=\{(y(t),u(t))\in \rline^2|y=\Phi(u,y_0),t\in \rline_+\}\subset D$ holds for all $u\in AC(\rline_+)$ and $(y_0,u(0))\in D$. For example, the set $D$ for the Dahl model in Section \ref{hysmodel} is given by $D=(-F_C, F_C)\times \rline$. Using $D$, we can generalize Theorem \ref{thm1} as follows.

\begin{proposition}\label{prop2}
Consider the Duhem hysteresis operator $\Phi$ defined in \rfb{babuskamodel} and \rfb{duhemf1f2} with $C^1$ functions $F, G:D \to \rline$ and with the traversing function $\omega_{\Phi}$ and the anhysteresis function $f_{an}$. Assume that the anhysteresis curve is in $D$ and there exists an intersecting function $\Lambda$ (e.g., the hypotheses in Lemma \ref{lemma_ass_intersect} hold). Assume further that the Assumption {\bf (A)} holds for all $(\sigma,\xi)$ in $D$. Then for every $u\in AC(\rline_+)$ and $(y_0,u(0))\in D$, the function $t\rightarrow H_{\circlearrowright}(y(t),u(t))$ with $H_{\circlearrowright}$ as in \rfb{storage_cw} and $y:=\Phi(u,y_0)$ is right differentiable and satisfies \rfb{passivehysteresis_cw}. Moreover, if the anhysteresis function $f_{an}$ satisfies $f_{an}(0)= 0$, then $H_{\circlearrowright}\geq 0$ and the Duhem operator is clockwise (CW).
\end{proposition}

The proof follows the same arguments as that of Theorem \ref{thm1}.


\section{Examples}
\subsection{The function $H_{\circlearrowright}$ for the Dahl model}
Recall the Dahl model as defined in Section \ref{dahlsec} and consider the case when $r =1$. In this case, the Dahl model can be reformulated into the Duhem operator as in \rfb{babuskamodel} with
\begin{equation}\label{f1dahl}
f_1(\sigma,\xi) = \rho \left ( 1-\frac{\sigma}{F_c}\right ),\ f_2(\sigma,\xi) = \rho \left ( 1+\frac{\sigma}{F_c}\right ),
\end{equation}
where $\rho >0$ and $F_c>0$. It is immediate to check that the conditions as given in \rfb{exis_cond} are satisfied, which means there exists solution for this Duhem operator for all positive time. The anhysteresis function of the Dahl model is $f_{an}(\xi) = 0$.

Calculating the curve $\omega_{\Phi}$, we have
\begin{equation}\label{omgdahl}
\omega_{\Phi}(\tau,y(t),u(t))=\left \{ \begin{array}{ll} F_c+(y(t)-F_c)e^{\frac{\rho}{F_c}(u(t)-\tau)} & \ \tau\in[u(t),\ \infty),\\
                              -F_c+(y(t)+F_c)e^{\frac{\rho}{F_c}(\tau -u(t))} & \ \tau\in(-\infty,\ u(t)]. \end{array} \right.
\end{equation}
From \rfb{f1dahl} and \rfb{omgdahl}, it is immediate to see that the pair $(y,u)$ is well-defined in $D=(-F_C,F_C)\times \rline$.
The intersecting function $\Lambda(y(t),u(t))$ is given as follows
\begin{equation}\label{intersectionfun}
\Lambda(y(t),u(t))=\left \{ \begin{array}{ll} u(t)+\frac{F_c}{\rho}\ln{\frac{F_c}{y(t)+F_c}} & \ y(t) \geq 0,\\
                              u(t)-\frac{F_c}{\rho}\ln{\frac{-F_c}{y(t)-F_c}} & \ y(t) < 0. \end{array} \right.
\end{equation}
Since $f_1$ and $f_2$ in \rfb{f1dahl} satisfy the Assumption {\bf (A)} for all $(\sigma,\xi)\in D$, the result in Proposition 2 holds.

By denoting $u^*(t)=\Lambda(y(t),u(t))$, we can compute explicitly the function $H_{\circlearrowright}$ as follows

%
\begin{align}\nonumber
H_{\circlearrowright}(y(t),u(t))&=\left \{ \begin{array}{ll} -F_c(u(t)-u^*(t))+\frac{F_c}{\rho}(y(t)+F_c)(1-e^{\frac{\rho}{F_c}(u^*(t)-u(t))}) & \ y(t) \geq 0\\
                               F_c(u(t)-u^*(t))+\frac{F_c}{\rho}(y(t)-F_c)(e^{\frac{\rho}{F_c}(u(t)-u^*(t))}-1) & \ y(t) < 0 \end{array} \right.\\ \label{storageDahl}
                               &=\left \{ \begin{array}{ll} \frac{F_C^2}{\rho}\ln{\frac{F_c}{y(t)+F_c}} +\frac{F_C}{\rho}y(t)& \ y(t) \geq 0\\
                               \frac{F_C^2}{\rho}\ln{\frac{-F_c}{y(t)-F_c}} -\frac{F_C}{\rho}y(t)& \ y(t) < 0. \end{array} \right.
\end{align}
Indeed, it can be checked that $\dot{H}_{\circlearrowright}\leq \dot{u}(t)y(t)$.
\subsection{Stability analysis of a second-order mechanical system with hysteretic friction}
Now, let us consider an  example of a mechanical system with the Dahl friction model given by $m\ddot{x}+d\dot{x}+kx+\Phi(x)=0$, where $m>0$, $d>0$, $k>0$, the hysteresis operator $\Phi$ is given as in \rfb{f1dahl} with $\rho>0$ and $F_c>0$. As discussed before, the functions $f_1$ and $f_2$ satisfy the hypotheses of Proposition 2.

The state space representation of the system is given as follows
\begin{equation*}
\left.\begin{array}{rl}
\dot{x}_1&=x_2,\\
\dot{x}_2&=-\frac{k}{m}x_1-\frac{d}{m}x_2-\frac{x_3}{m},\\
\dot{x}_3&=\rho \left (1-\frac{x_3}{F_C}\right)x_{2+}+\rho\left (1+\frac{x_3}{F_C}\right)x_{2-}.\end{array} \right.
\end{equation*}
Using $V(x_1,x_2,x_3)=\frac{1}{2}kx_1^2+\frac{1}{2}mx_2^2+H_{\circlearrowright}(x_3,x_1)$, where $H_{\circlearrowright}$ is as in \rfb{storageDahl} and satisfies \rfb{passivehysteresis_cw}, a routine calculation shows that
\begin{align*}
\dot V & \leq  -x_2 x_3 -dx_2^2 + x_3x_2 \\
& = -dx_2^2.
\end{align*}
Since the relations of the corresponding Dahl operator \big(i.e. the set $\calR_{y_0,u}:=\{(y(t),u(t))|y=\Phi(u,y_0)\}$\big) is contained in $(-F_C,F_C)\times \rline$ for all $y_0 \in (-F_C,F_C)$ and $u \in AC(\rline_+)$, then it implies that $x_3$ (which is the output of the Dahl operator) is bounded and lies in the interval $(-F_C,F_C)$. Additionally, we have $V$ which is lower bounded and radially unbounded in the first and second arguments, i.e. $V(x_1,x_2,x_3)\rightarrow \infty$ as $\left \|\sbmn{x_1\\x_2} \right \|\rightarrow \infty$. Thus $\dot{V} \leq -dx_2^2$ implies that the state trajectory $(x_1,x_2)$ is bounded. Moreover, using the boundedness of $(x_1,x_2)$ and the boundedness of $x_3$, an application of the Lasalle's invariance principle shows that $(x_1,x_2,x_3)$ converges to the largest invariant set where $x_2=0$. By analyzing the corresponding state equations, this invariant set is given by $\{(x_1,x_2,x_3)|kx_1=-x_3,\ x_2=0\}$.

\section{Conclusion}
In this paper, we have investigated the clockwise I/O dynamics of a class of Duhem hysteresis operator by obtaining sufficient conditions for the Duhem operators to be CW. The CW property is obtained via the construction of a suitable function satisfying the CW dissipation inequality which can be useful for studying stability of systems having CW hysteretic element, such as, mechanical systems with hysteretic friction. The sufficient conditions for CW I/O dynamics incorporates also the knowledge of anhysteresis function which is commonly neglected in the literature of hysteretic systems. For systems identification of hysteresis systems, the results provide additional characterization of the Duhem operators that can be used to restrict the class of the Duhem operators which will be fitted with the measurement data.

\bibliographystyle{model1b-num-names}

\appendix
\section{Proof of Lemma \ref{lemma_ass_intersect}}\label{prooflemmaintersect}
\begin{proof}
The proof is similar to the proof of \cite[Lemma 3.1]{JAYAWARDHANACDC2011}. Consider the continuous function $\varphi:\rline^3\rightarrow\rline$ defined as $\varphi(\xi,y_0, u_0) =  \omega_\Phi(\xi,y_0,u_0)-f_{an}(\xi)$.
Consider also $A_0$ and $A_1$ the two subsets of $\rline^3$ defined as,
\begin{eqnarray*}
&A_0= \{(\xi,y_0,u_0)\in\rline^3, \ y_0>f_{an}(u_0)\ ,\ \xi< u_0\}\ ,\\
&A_1= \{(\xi,y_0,u_0)\in\rline^3, \ y_0<f_{an}(u_0)\ ,\ \xi> u_0\}\ .
\end{eqnarray*}
Note that the function $f_{an}$ being strictly increasing by assumption, implies that these sets are open sets.
Also, the function $\omega_\Phi$ satisfies
\begin{align*}
\frac{\partial \omega_\Phi}{\partial \xi}(\xi,y_0,u_0) &= f_2(\omega_\Phi(\xi,y_0,u_0),\xi)\qquad \forall (\xi,y_0,u_0)\ \in\ A_0\ ,\\
\frac{\partial \omega_\Phi}{\partial \xi}(\xi,y_0,u_0) &= f_1(\omega_\Phi(\xi,y_0,u_0),\xi) \qquad \forall (\xi,y_0,u_0)\ \in\ A_1\ .
\end{align*}
Consequently, $\omega_\Phi(\xi,y_0,u_0)$ is the solution of ordinary differential equations computed from $C^1$ vector field. With \cite[Theorem V.3.1]{HARTMAN1982}, it implies that $\omega_\Phi$ is a $C^1$ function in $A_0\cup A_1$.
Moreover, the function $f_{an}$ being $C^1$ implies that the function $\varphi$ is $C^1$ in $A_0\cup A_1$.
With \rfb{eq_ass_Lambda_1} and (\ref{eq_ass_Lambda_2}), the function $\varphi$ satisfies,
\begin{eqnarray*}
\frac{\partial \varphi}{\partial \xi}(\xi,y_0,u_0)  &>& \epsilon \neq 0\ ,\ \forall (\xi,y_0,u_0)\ \in\ A_0\cup A_1\ .
\end{eqnarray*}
Consequently, $\varphi$ is a strictly increasing function in its first argument in the set $A_0\cup A_1$.
This also implies that,
\begin{align*}
\varphi(\xi,y_0,u_0) &< \varphi(u_0,y_0,u_0) + \epsilon(\xi-u_0)\ \\ & \qquad \forall (\xi,y_0,u_0)\ \in\ A_0\ ,\\
\varphi(\xi,y_0,u_0) &> \varphi(u_0,y_0,u_0) + \epsilon(\xi-u_0)\ \\ & \qquad \forall (\xi,y_0,u_0)\ \in\ A_1\ .
\end{align*}
Note that if $y_0>f_{an}(u_0)$, then $\varphi(u_0,y_0,u_0)>0$ and consequently there exists a unique real number $u^*$ such that $\varphi(u^*,y_0,u_0)=0$ and $(u^*,y_0,u_0)\in A_0$. On the other hand, if $y_0<f_{an}(u_0)$, then $\varphi(u_0,y_0,u_0)<0$ and consequently there exists a unique real number $u^*$ such that $\varphi(u^*,y_0,u_0)=0$ and $(u^*,y_0,u_0)\in A_1$. Therefore, by denoting $\Lambda(y_0,u_0)=u^*$, by employing the implicit function theorem and using the fact that $\varphi$ is $C^1$, it can be shown that $\Lambda$ is $C^1$.
\end{proof}
\section{Proof of Theorem \ref{thm1}}\label{proofofthm1}
\begin{proof}
The proof of Theorem \ref{thm1} follows the same line as in our previous work \cite{JAYAWARDHANACDC2011}. In the first part of the proof we will prove that for all $t\in\rline_+$, $\dot H_{\circlearrowright}\big(y(t), u(t)\big)$ exists and satisfies \rfb{passivehysteresis_cw}. In the second part we show the non-negativeness of $H_{\circlearrowright}\big(y(t), u(t)\big)$.

To show that $H_{\circlearrowright}$ exists, let us denote $u^*:=\Lambda(y,u)$. Using the Leibniz derivative rule, we have
\begin{align}\nonumber
\frac{\dd}{\dd t}H_{\circlearrowright}(y(t),u(t)) & = \frac{\dd}{\dd t}\left[ \int_{0}^{\Lambda(y(t), u(t))}{f_{an}(\tau) \dd \tau}-\int_{u(t)}^{\Lambda(y(t),u(t))}{\omega_\Phi(\tau,y(t),u(t)) \dd \tau}\right]\\\nonumber
&=\dot{\Lambda}(y(t),u(t))f_{an}(\Lambda(y(t),u(t)))-\dot{\Lambda}(y(t),u(t))\omega_{\Phi}(\Lambda(y(t),u(t)),y(t),u(t))\\\nonumber
&+\dot{u}(t)\omega_{\Phi}(u(t),y(t),u(t)) -\int_{u(t)}^{\Lambda(y(t),u(t))}{\frac{\dd \omega_{\Phi}(\tau,y(t),u(t))}{\dd t}\dd \tau}\\\label{Hdot}
&=\dot{u}(t) y(t) - \int_{u(t)}^{u^*(t)}{\frac{\dd}{\dd t} \omega_\Phi(\tau,y(t),u(t)) \dd\tau},
\end{align}
where $u^*(t)=\Lambda(y(t),u(t))$ and the last equation is due to $\omega_{\Phi}(u(t),y(t),u(t))=y(t)$ and by the hypothesis given in Lemma \ref{lemma_ass_intersect}.
The first term in the RHS of
\rfb{Hdot} exists for all $t\geq 0$ since $u(t)$ satisfies
\rfb{babuskamodel}.
In order to get \rfb{passivehysteresis_cw}, it remains to check whether
the last term of \rfb{Hdot} exists, is finite and satisfies
\begin{equation}\label{posofw}
\int_{u(t)}^{u^*(t)}{\frac{\dd}{\dd t} \omega_\Phi(\tau,y(t),u(t)) \dd\tau}\geq 0.
\end{equation}
It suffices to show that, for every
$\tau\in [u(t), u^*(t)]$, the following limit
\begin{equation}\label{rderivativeofw_anhys}
\lim_{\epsilon\searrow 0^+} \frac{1}{\epsilon} [\omega_\Phi(\tau,
y(t+\epsilon), u(t+\epsilon))-\omega_\Phi(\tau, y(t), u(t))]
\end{equation}
exist and the limit of \rfb{rderivativeofw_anhys} is greater or equal to zero when $u^*(t)>u(t)$ and the
limit is less or equal to zero elsewhere.


For any $\epsilon\geq 0$, let us introduce the continuous function
$\omega_\epsilon : \rline \rightarrow \rline$ by
\begin{equation}\label{definewe}
\omega_\epsilon(\tau)= \omega_\Phi(\tau, y(t+\epsilon),
u(t+\epsilon)).
\end{equation}
More precisely, using \rfb{wcurve}, $\omega_{\epsilon}$ is
the unique solution of
\begin{equation}\label{wh}
\omega_\epsilon(\tau)= \left\{
\begin{array}{l}
y(t+\epsilon) + \displaystyle\int_{u(t+\epsilon)}^\tau{f_1( \omega_\epsilon(s), s) \dd s} \quad \forall\tau \geq u(t+\epsilon)\\[1.5em]
y(t+\epsilon) + \displaystyle\int_{u(t+\epsilon)}^\tau{f_2(
\omega_\epsilon(s), s) \dd s} \quad
\forall\tau \leq u(t+\epsilon).
\end{array}\right.
\end{equation}
Note that $\omega_0(\tau)=\omega_\Phi(\tau, y(t), u(t))$ as in \rfb{wcurve} for all $\tau\in\rline$ and
\begin{equation}
\omega_\epsilon(u(t+\epsilon)) = y(t+\epsilon) \qquad \forall \;
\epsilon\;\in\;\rline_+\ .\label{whproperty}
\end{equation}

In order to show the existence of \rfb{rderivativeofw_anhys} and the validity of \rfb{posofw}, we consider
several cases depending on the sign of $\dot u(t)$ and $F(y(t),u(t))$.
It can be checked that the hypothesis {\bf{(A)}} on $F$ implies that
$f_1(y(t),u(t))\geq f_2(y(t),u(t))$ whenever $y(t) \leq f_{an}(u)$, and $f_1(y(t),u(t))< f_2(y(t),u(t))$ otherwise.

First, we assume that $\dot u(t) > 0$ and $y(t) \geq f_{an}(u(t))$. In this case, according to Lemma \ref{lemma_ass_intersect}, we have $u^*(t) < u(t)$. Since $\dot u(t) > 0$, there exists $\gamma > 0$ such that $\tau
\leq u(t)< u(s) $ for all $s$ in $(t,
t+\gamma)$. It follows from \rfb{wh} and assumption ({\bf A}) that
for every $\epsilon\in (0,\gamma)$:

\begin{align*}
\nonumber\frac{\dd \omega_\epsilon(u(s))}{\dd s} & =  f_2( \omega_\epsilon(u(s)), u(s))\;\dot u(s) \\
& \geq f_1( \omega_\epsilon(u(s)), u(s))\;\dot u(s) \quad \forall s \in [t,t+\epsilon],
\end{align*}
and the function $\omega_0$ satisfies
$$
\frac{\dd \omega_0(u(s))}{\dd s}\;=\;f_1( y(s), u(s))\;\dot u(s) \qquad
\forall  s\;\in\;[t, t+\epsilon].
$$
Since the functions $\epsilon \mapsto w_0(u(t+\epsilon))$ and $\epsilon
\mapsto y(t+\epsilon)$ with $\epsilon \in (0,\gamma]$ are two $C^1$
functions which are solutions of the same locally Lipschitz ODE and
with the same initial value. By uniqueness of solution, we get
$\omega_0(u(t+\epsilon)) = y(t+\epsilon)$.

This together with the fact that $\omega_\epsilon(u(t+\epsilon)) = y(t+\epsilon)$ and using the
comparison principle (in reverse direction), we get that for every
$\epsilon\in [0,\gamma)$:
$$
\omega_\epsilon(u(s)) \;\leq\; \omega_0(u(s)) \qquad \forall \; s\;\in\;[t,
t+\epsilon].
$$
Since the two functions $\omega_\epsilon(\tau)$ and $\omega_0(\tau)$
for $\tau \in [u^*(t),u(t)]$ are two solutions of the same ODE, it
follows that \footnote{Otherwise there exist $\tau_1 <\tau_2$ such
that $\omega_\epsilon(\tau_1)\;=\; \omega_0(u(\tau_1))$ and
$\omega_\epsilon(\tau_2)\;>\; \omega_0(u(\tau_2))$ which contradict
the uniqueness of the solution of the locally Lipschitz ODE.}
$\omega_\epsilon(\tau)\;\geq\; \omega_0(\tau)$ 
and we get that if it exists:
\begin{equation}\label{limitofwhwo_anhys}
\lim_{\epsilon\searrow 0^+} \frac{1}{\epsilon}
[\omega_\epsilon(\tau)-\omega_0(\tau)] \;\leq\; 0 \qquad \forall
\tau \in [u^*(t),u(t)] .
\end{equation}
Then it is clear that
\begin{equation}\label{limitofwhwo_anhys1}
\lim_{\epsilon\searrow 0^+} \frac{1}{\epsilon}
[\omega_\epsilon(\tau)-\omega_0(\tau)] \;\geq\; 0 \qquad \forall
\tau \in [u(t),u^*(t)] .
\end{equation}
In the following, we show the existence of the limit given in \rfb{limitofwhwo_anhys} by computing a bound on the function
$\epsilon\mapsto \frac{1}{\epsilon} [\omega_\epsilon(\tau)-\omega_0(\tau)]$. Note that for every $\epsilon\in [0,\gamma]$,
\begin{align*}
& |\omega_\epsilon(\tau)-\omega_0(\tau)| \leq |y(t+\epsilon)-y(t)| + \left|\int_{u(t+\epsilon)}^{u(t)}f_2(\omega_\epsilon(s),s)\,\dd s\right|  \\ & \qquad \qquad + \left|\int_{u(t)}^\tau f_2(\omega_\epsilon(s),s)-f_2(\omega_0(s),s)\,\dd s \right| \\
& \leq |y(t+\epsilon)-y(t)| +
\int_{u(t)}^{u(t+\epsilon)}|f_2(\omega_\epsilon(s),s)|\,\dd s \\ & \qquad \qquad +
\int_{\tau}^{u(t)}|f_2(\omega_\epsilon(s),s)-f_2(\omega_0(s),s)|\,\dd
s,
\end{align*}
for all $\tau \in [u^*(t),u(t)]$. By the locally Lipschitz property of
$f_2$ and by the boundedness of
$\omega_\epsilon$ on $[\tau,u(t)]$ for all $\epsilon\in [0,\gamma]$, it can be shown that there exists $\alpha$, such that $\alpha$ is a bound of $f_2$ on a compact set. Then
\begin{multline*}
|\omega_\epsilon(\tau)-\omega_0(\tau)|\leq |y(t+\epsilon)-y(t)| \\ +
\int_{\tau}^{u(t)}\,L\,|\omega_\epsilon(s)-\omega_0(s)|\,\dd s +
\alpha|u(t+\epsilon)-u(t)|\ ,
\end{multline*}
where $L$ is the Lipschitz constant of $f_2$ on $[\omega_{\min{}},
\omega_{\max{}}]\times [\tau,u(t)]$ with
\begin{align*}
\omega_{\min{}} & = \min_{(c,s)\in[0,\gamma]\times[\tau, u(t)]}\omega_c(s), \\
\omega_{\max{}} & = \max_{(c,s)\in[0,\gamma]\times[\tau,
u(t)]}\omega_c(s)\ .
\end{align*}
With Gronwall's lemma, this implies that for every $\epsilon\in [0,\gamma]$
\begin{multline*}
|\omega_\epsilon(\tau)-\omega_0(\tau)| \leq
\exp((u(t)-\tau)L)\Big[|y(t+\epsilon)-y(t)| +
\alpha|u(t+\epsilon)-u(t)|\Big],
\end{multline*}
for all $\tau \in [u^*(t),u(t)]$. Hence
\begin{multline*}
\lim_{\epsilon\searrow 0^+} \frac{1}{\epsilon}
|\omega_\epsilon(\tau)-\omega_0(\tau)| \leq
\exp((u(t)-\tau)L)\Big[|f_1(y(t),u(t))|+\alpha\Big]\,\dot u(t),
\end{multline*}
for all $\tau \in [u^*(t),u(t)]$.
Consequently the limit given in (\ref{limitofwhwo_anhys}) exists.
It implies that the inequality \rfb{posofw} holds when $\dot u(t)>0$ and $y(t)\geq f_{an}(u(t))$.

For the next case, we assume that $\dot u(t) > 0$ and $y(t) < f_{an}(u(t))$. Again, according to Lemma \ref{lemma_ass_intersect}, we have $u^*(t)>u(t)$. Since for every $\epsilon\in (0,\gamma]$ the two functions $\omega_\epsilon(\tau)$
and $\omega_0(\tau)$ satisfy the same ODE for\footnote{we have for
all $\tau \in [u(t+\epsilon),u^*(t)]$~:
$$\frac{\dd \omega_\epsilon(\tau)}{\dd\tau} \;=\; f_1( \omega_\epsilon(\tau), \tau) \quad, \qquad
\frac{\dd \omega_0(\tau)}{\dd\tau} \;=\; f_1( \omega_0(\tau),
\tau)$$ } $\tau \in [u(t+\epsilon),u^*(t)]$, we have
$$
\omega_\epsilon(\tau) = \omega_0(\tau) \qquad \forall \tau \in
[u(t+\epsilon),u^*(t)],
$$
for all $\epsilon\in [0,\gamma]$. This implies that
\begin{equation}\label{proofeq}
\lim_{\epsilon\searrow 0^+} \frac{1}{\epsilon}
[\omega_\epsilon(\tau)-\omega_0(\tau)] \;=\; 0.
\end{equation}
We can use similar arguments to prove that \rfb{posofw} is satisfied when $\dot u(t) < 0$.

Finally, when $\dot u(t)=0$, we simply get $$\lim_{\epsilon\searrow
0^+} \frac{1}{\epsilon} |\omega_\epsilon(\tau)-\omega_0(\tau)|=0,$$
by continuity of the above bound.


For the second step, we need to show that $H_{\circlearrowright}$ is non-negative. Consider the case when $y(t) \geq f_{an}(u(t))$,  we have $u^*(t) < u(t)$ and $\omega_{\Phi}(\tau) \geq f_{an}(\tau)$ for all $\tau \in [u^*(t),  u(t)]$ by Lemma \ref{lemma_ass_intersect}. Since $f_{an}(\tau)$ belongs to the sector $[0 ,\ \infty)$ for all $\tau \in \rline$, we have
\begin{equation*}
H_{\circlearrowright}(y(t),u(t))=\int_{0}^{u(t)}{f_{an}(\tau) \dd \tau}+\int_{u(t)}^{u^*(t)}{f_{an}(\tau)-\omega_\Phi(\tau,y(t),u(t)) \dd \tau} \geq 0.
\end{equation*}
In case when $y(t) < f_{an}(u(t))$, we can show the non-negativeness of $H$ by using similar arguments.
\end{proof}
\end{document}